\documentclass[a4paper,12pt,oneside,dvipsnames,reqno]{amsart} 

\usepackage{amsmath,amsfonts,amsthm,amssymb}

\usepackage{xcolor}
\usepackage{graphicx}
\usepackage{cite}
\usepackage{mathdots}

\usepackage{bbm}
\usepackage{tikz-cd}

\usepackage{hyperref}
\hypersetup{colorlinks=true, citecolor=PineGreen, linkcolor=RoyalBlue, linktoc=page,urlcolor=RoyalBlue}

\usepackage{geometry}

\theoremstyle{plain}
\newtheorem{theorem}{Theorem}[section] 
\newtheorem{lemmy}[theorem]{Lemma}
\newtheorem{prop}[theorem]{Proposition}
\newtheorem{cor}[theorem]{Corollary}

\theoremstyle{remark}
\newtheorem{rem}[theorem]{Remark}

\theoremstyle{definition}

\newcommand{\R}{\mathbb{R}}

\newcommand{\C}{\mathbb{C}}

\newcommand{\Z}{\mathbb{Z}}

\newcommand{\N}{\mathbb{N}}

\newcommand{\abs}[1]{\lvert{#1}\rvert}

\title{An observation concerning highly ramified $\epsilon$-factors}

\author{Edgar Assing}

\newcommand{\Date}{16$^{\mathrm{th}}$ September 2024}
\date{\Date}

\email{assing@math.uni-bonn.de}

\begin{document}

\thanks{The author is supported by the Germany Excellence Strategy grant EXC-2047/1-390685813 and also partially funded by the Deutsche Forschungsgemeinschaft (DFG, German Research Foundation) – Project-ID 491392403 – TRR 358.}

\begin{abstract}
In this note we prove a quantitative stability result for the $\epsilon$-factors associated to generic irreducible representations of $\textrm{GL}_n(F)$ under twists by highly ramified characters, where $F$ is a non-archimedean local field.
\end{abstract}

\maketitle

\section{Introduction}

Let $F$ be a non-archimedian local field equipped with a non-trivial additive character $\psi$. Given a complex irreducible smooth admissible representation $\pi$ of $\textrm{GL}_n(F)$ we can attach $L$- and $\epsilon$-factors $L(s,\pi)$ and $\epsilon(s,\pi,\psi)$, for $s\in \C$. These were constructed in \cite{godement_jacquet}, see also \cite{goldfeld_hundley}, and are known to satisfy a set of nice properties one of them being the following stability result:

\noindent\textbf{Stability of Local Factors.} {\it Let $\pi_1$ and $\pi_2$ be two complex irreducible smooth admissible representation $\pi$ of $\textrm{GL}_n(F)$ with the same central character. Then we have
\begin{equation}
	\epsilon(s,\chi\otimes \pi_1,\psi) = \epsilon(s,\chi\otimes \pi_2,\psi),\label{eq:stability_first_form}
\end{equation}
for all sufficiently ramified quasi-characters $\chi\colon F^{\times}\to \C^{\times}$.}

A proof of this in the present context is given in \cite{jacquet_shalika_ramified}. As stated this phenomenon is only the initial case of a more general notion of \emph{stability of $\gamma$-factors}. Such stability results are by now established in great generality and play a key role in the proofs of many known instances of Langlands functoriality. We refer to \cite{stab_root, shahidi2} and the references within for a more detailed discussion. 

An immediate caveat of the formulation given above is the vaguely defined notion \emph{sufficiently ramified}. More precisely one can formulate stability of local factors as follows. There is a constant $N=N(\pi_1,\pi_2)\in \N$ depending on $\pi_1$ and $\pi_2$ such that \eqref{eq:stability_first_form} holds for all quasi-characters $\chi$ with exponent conductor $a(\chi)$ larger than $N$. This raises the immediate question of what can be said about $N$. The answer is certainly well known to experts, who often point to \cite{De}. However, to the best of our knowledge a clean answer is not yet well recorded in the literature. We rectify this here:

\begin{theorem}\label{main_th}
Let $\pi$  be a complex generic irreducible smooth admissible representation $\pi$ of $\textrm{GL}_n(F)$ with central character $\omega$ and exponent conductor $a(\pi)$. Then, for every quasi-character $\chi$ with $a(\chi)\geq a(\pi)$ we have
\begin{equation}
	\epsilon(s,\chi\otimes \pi,\psi) = \epsilon(s,\omega\chi,\psi) \cdot \epsilon(s,\chi,\psi)^{n-1}. \label{eq:main_th}
\end{equation}
\end{theorem}

\begin{cor}\label{main_cor}
Let $\pi_1$ and $\pi_2$ be two complex generic irreducible smooth admissible representations $\pi$ of $\textrm{GL}_n(F)$ with the same central character. Then \eqref{eq:stability_first_form} holds for all quasi-characters $\chi$ with $a(\chi)\geq \max(a(\pi_1),a(\pi_2))$.
\end{cor}

Such quantitative stability results have applications in analytic number theory. For example, it was observed by Duke and Iwaniec in \cite[(7)]{duke_iwaniec} that hyper Kloosterman sums can be expressed as a sum involving powers of $\textrm{GL}_1$ $\epsilon$-factors. To illustrate this we will use our stability result to express certain local Bessel transforms in terms of hyper Kloosterman sums. To do so let $\mathcal{O}$ be the ring of integers in $F$, let $\mathfrak{p}=(\varpi)$ be its unique maximal ideal and put $q=\sharp \mathcal{O}/\mathfrak{p}$. Further, assume that $\psi$ is trivial on $\mathcal{O}$ but non-trivial on $\mathfrak{p}^{-1}$. The Bessel transform associated to a representation $\pi$ has been defined in \cite{ichino_templier, corbett_gln} and we denote it by $\Phi\mapsto \mathcal{B}_{\pi}\Phi$. See \eqref{eq:bessel_trafro} below for the usual characterization. We have the following result:

\begin{cor}\label{cor:bessel}
Let $\pi$  be a complex generic irreducible smooth admissible representation $\pi$ of $\textrm{GL}_n(F)$ with trivial central character and $a(\pi)>1$. Let $$\Phi(x) = \psi(xz_0\varpi^{-t})\mathbbm{1}_{\mathcal{O}^{\times}}(x),$$ for $t\geq a(\pi)$ and $z_0\in \mathcal{O}^{\times}.$ Then we have
\begin{multline*}
	[\mathcal{B}_{\pi}\Phi](y_0\cdot \varpi^{-nt}) = \mathbbm{1}_{\mathcal{O}^{\times}}(y_0)\cdot \frac{q^{\frac{t}{2}(n-4)(n-2)}}{(1-q^{-1})^{n-1}} \\ \cdot \sum_{x_1,\ldots,x_{n-2}\in (\mathcal{O}/\mathfrak{p}^t)^{\times}} \psi\left(\varpi^{-t}\left[x_1+\ldots+x_{n-2}+\frac{y_0z_0^{-1}}{x_1\cdots x_{n-2}}\right]\right).
\end{multline*}
\end{cor}

This corollary is related to computations in \cite[Section~4]{corbett_gln} and can be inserted directly in the Vorono\"i summation formula given in \cite[Theorem~1.1]{corbett_gln}. In particular, at places where the additive twist is highly ramified, we find $(n-1)$-hyper Kloosterman sums as expected. 

The plan of this paper is the following. We start by recalling common notation and some preliminary results in Section~\ref{sec:prelim} below. The proof of Theorem~\ref{main_th} is contained in Section~\ref{sec:proof}. Classification results allow us to reduce to the case of supercuspidal representations. The latter can be treated by combining the Jacquet-Langlands correspondence with the theory of non-abelian Gau\ss\  sums developed by Bushnell and Fr\"ohlich. Finally, in Section~\ref{sec:bessel} we apply Theorem~\ref{main_th} to compute certain Bessel transforms. The main result here is Proposition~\ref{prop:bessel}, which directly implies Corollary~\ref{cor:bessel}.

\section{Notation and preliminaries}\label{sec:prelim}

Throughout let $F$ be a non-archimedean local field with ring of integers $\mathcal{O}$. We denote the unique maximal ideal by $\mathfrak{p}$ and write $q$ for the residual characteristic (i.e. $q=\sharp \mathcal{O}/\mathfrak{p}$). Once and for all we fix a uniformizer $\varpi\in F^{\times}$ and normalize the valuation $v$ of $F$ by $v(\varpi)=1$. The corresponding absolute value on $F$ is given by $\vert x\vert = q^{-v(x)}$. Finally, we normalize the Haar measure on $F$ such that $\textrm{Vol}(\mathcal{O}) = 1$.

We pick a non-trivial additive character $\psi\colon F\to S^1$. For $a\in F^{\times}$, we define
\begin{equation}
	\psi_a(x) = \psi(ax) \text{ for }x\in F. \nonumber
\end{equation}
The conductor $n(\psi)$ of $\psi$ is given by $n(\psi) = \min\{n\in\Z\colon \psi\vert_{\mathfrak{p}^n}\equiv 1\}$. Observe that $n(\psi_a) = n(\psi)-v(a)$.

Further, we will encounter a finite dimensional central division algebra $D$ of dimension $n^2$ over $F$. Let $\mathcal{O}_D$ be the maximal order in $D$ with maximal ideal $\mathfrak{P}_D$. The surjective valuation is denoted by $v_D\colon D^{\times}\to \Z$ and we write $\mathfrak{k}_D = \mathcal{O}_D/\mathfrak{P}_D$. In the multiplicative group $D^{\times}$ we define a filtration of subgroups by setting $U_D(0)=\mathcal{O}_D^{\times}$ and $U_D(m)=1+\mathfrak{P}_D^m$.

We write $\textrm{Nrd}(\cdot)$ and $\textrm{Trd}(\cdot)$ for the reduced norm and trace respectively. The trace is used to lift the additive character $\psi$ on $F$ to $D$ by setting
\begin{equation}
	\psi_D(x) = \psi(\textrm{Trd}(x)), \text{ for }x\in D.\nonumber
\end{equation}
This character can be used in the usual way to identify $D$ with its Pontrjagin dual $\widehat{D}$.

\subsection{Multiplicative characters}

A (multiplicative) quasi-character is a homomorphism $\chi\colon F^{\times} \to \C^{\times}$. Let $\mathfrak{X}$ denote the set of characters $\xi$ satisfying $\chi(\varpi)=1$. These are uniquely determined by their restriction to $\mathcal{O}^{\times}$ and are of finite order. Every quasi-character is of the form $$x\mapsto \chi(x)\vert x\vert^s,$$ for $\chi\in \mathfrak{X}$ and $s\in \C$. Given a quasi-character $\chi$ we define the conductor exponent $a(\chi)$ by $a(\chi) =0$ if $\chi\vert_{\mathcal{O}^{\times}}=1$ and by
\begin{equation}
	a(\chi) = \min \{ a\in \N \colon \chi\vert_{1+\mathfrak{p}^a}\equiv 1\}. \label{eq:achi}
\end{equation}
If $a(\chi) =0$, we call $\chi$ unramified. In this case $\chi=\abs{\cdot}^s$ for some $s\in \C$.

Let $\chi$ be a quasi-character of $F^{\times}$. In this situation the local $L$- and $\epsilon$-factors were defined in \cite{tate_thesis}, see also \cite{tate_background}. We will recall basic properties of these in Section~\ref{sec:loc_fac} below.

We will not recall the general construction of $\epsilon$-factors. However, it will be useful to record their relation to classical Gau\ss\  sums. To see this we assume that $n(\psi)=0$ and take $\chi\in \mathfrak{X}$. Recall that the classical Gau\ss\  sum is defined by
\begin{equation*}
	\tau(\chi) = \sum_{x\in \mathcal{O}^{\times}/(1+\mathfrak{p}^{a(\chi)})} \chi(x)\psi(x\varpi^{-a(\chi)}) = q^{a(\chi)}\int_{\mathcal{O}^{\times}} \chi(x)\psi(x\varpi^{-a(\chi)})dx.
\end{equation*}
The root number $W(\chi)$ is then given by
\begin{align}
		W(\chi) &= \chi(-1)\overline{\tau(\chi)} q^{-\frac{a(\chi)}{2}} = \tau(\chi^{-1})q^{-\frac{a(\chi)}{2}} \nonumber\\
		&= \epsilon(\frac{1}{2},\chi,\psi).\label{eq:gl_2_rootnb}
\end{align}

Recall that for every $\chi\in \mathfrak{X}$ there is $v_{\chi}\in \mathcal{O}^{\times}$ such that
\begin{equation}
	\chi(1+u\varpi^{\lceil\frac{a(\chi)}{2}\rceil}) = \psi(uv_{\chi}\varpi^{n(\psi)-\lfloor{\frac{a(\chi)}{2}}\rfloor}) \label{eq:def_v_chi}
\end{equation}
for all $u\in \mathcal{O}$. (We hide the dependence of $v_{\chi}$ on $\psi$ in the notation, because we view the additive character as fixed.) With this notation at hand we can recall the following important result, which can be seen as seed for the general stability results:

\begin{lemmy}[Section~1 Corollary~2, \cite{tate_local}] \label{lm_gl1_stab}
Let $\mu,\chi\in \mathfrak{X}$ be characters such that $2a(\mu)\leq a(\chi)$. Then $$\epsilon(\frac{1}{2},\mu\chi,\psi) = \mu(v_{\chi})\epsilon(\frac{1}{2},\chi,\psi),$$ for $v_{\chi}$ as defined in \eqref{eq:def_v_chi}.
\end{lemmy}

We turn towards characters of $D^{\times}$. If $\xi$ is a character of $D^{\times}$, then we define its conductor exponent by
\begin{equation*}
	a_D(\chi) = \min\{ n\in \Z_{\geq 0}\colon \xi\vert_{U_D(n)}\equiv 1\}.
\end{equation*}
Note that all characters of $D^{\times}$ have the form $\chi\circ \textrm{Nrd}$, where $\chi$ is a quasi-character of $F^{\times}$. Thus, we will need a brief dictionary translating between properties of $\chi$ and $\chi\circ\textrm{Nrd}$. First, we compare conductor exponents. These are related as follows: 
\begin{equation}
	a_D(\chi\circ\textrm{Nrd})+n-1 = na(\chi) \label{cond_comp} \nonumber
\end{equation}
This is for example spelled out in \cite[Lemma~3.3]{corbett_twist} or can be deduced from \cite[Proposition~4.1.2]{gauss_sums}. Second, we need to lift \eqref{eq:def_v_chi} through the norm. The desired statement is given in \cite[Proposition~4.1.3]{gauss_sums} and we record it in the form
\begin{equation}
	\chi(\textrm{Nrd}(x)) = \psi_D(xv_{\chi}\varpi^{n(\psi)-\lfloor{\frac{a(\chi)}{2}}\rfloor}), \label{eq:additiviation}
\end{equation}
where $v_{\chi}$ is as in \eqref{eq:def_v_chi}. Finally, we define the root number by
\begin{equation}
	W(\chi\circ\textrm{Nrd}) = W(\chi)^n. \label{eq:imp_rel_chi_norm}
\end{equation}
This definition will be justified in Remark~\ref{rem:justified} below once we have introduced more general root numbers.

\subsection{Representations and conductors}

We now briefly recall some relevant aspects of the representation theory of ${\rm GL}_n(F)$ and $D^{\times}$.  All our representations $\pi$ will be smooth admissible and irreducible acting on some complex vector space. See for example \cite{bernstein_ze_0, goldfeld_hundley} for an introduction to these notions in the context of $\textrm{GL}_n(F)$ and \cite[Section~2.1]{gauss_sums} for the relevant definitions over $D^{\times}$. An irreducible representation comes with a central character denoted by $\omega_{\pi}$. We always view this as quasi-character of $F^{\times}$.

\subsubsection{The $\textrm{GL}_n(F)$ case}

We continue with some more discussion on representations $\pi$ of $\textrm{GL}_n(F)$. For a quasi-character $\chi$ we will write $\chi\cdot \pi= \chi\otimes \pi$ for the twist of $\pi$ by $\chi$. The action is explicitly given by $[\chi\cdot \pi](g) = \chi(\det(g))\cdot \pi(g)$. We obviously have $\omega_{\chi\cdot \pi} = \omega_{\pi} \chi^n$.  

A representation $\pi$ of $\textrm{GL}_n(F)$ is said to be unramified (or spherical) if its restriction to $\textrm{GL}_n(\mathcal{O})$ contains the trivial representation. Further, we say that $\pi$ is (essentially) square integrable if there is $s\in \R$ such that $\vert\cdot \vert^s\cdot \pi$ is unitary and has square integrable matrix coefficients. Similarly we call $\pi$ (essentially) supercuspidal if, up to unramified twist, it has compactly supported matrix coefficients. Finally, we say that $\pi$ is generic, if it features a non-trivial Whittaker functional. 

Finally, given a generic representation $\pi$ we associate the conductor exponent $a(\pi)$ as in \cite{jacquet_cond, jacquet_corr}. More precisely, if 
\begin{equation*}
	K_0(n) = \{ k = (k_{i,j})_{1\leq i,j\leq n}\in  \textrm{GL}_n(\mathcal{O}) \colon (k_{n,j})_{1\leq j\leq n}\equiv (0,\ldots,0,1) \text{ mod }\mathfrak{p}^n  \}
\end{equation*}
denotes the Hecke congruence subgroup of $K_0(0) = \textrm{GL}_n(\mathcal{O})$ as defined in \cite[p.211]{jacquet_cond}, then we have
\begin{equation*}
	a(\pi) = \min \{n\in \Z_{\geq 0}\colon \pi\vert_{K_0(n)} \text{ contains the trivial representation} \}. \nonumber
\end{equation*}
This is a natural extension of \eqref{eq:achi}. Note that, if $\pi$ is unramified and generic, then $a(\pi)=0$. For later reference we recall that for essentially square integrable $\pi$ we have the useful inequality
\begin{equation}
	a(\omega_{\pi})\leq \frac{a(\pi)}{n}.\label{eq:useful_cond}
\end{equation}
This is for example stated in \cite[Proposition~2.4]{corbett_twist}.

\subsubsection{The $D^{\times}$ case}

We move on to the discussion of representations $\sigma$ of $D^{\times}$. Write $\mathcal{A}_D(\omega)$ for the collection of (equivalence classes of) smooth admissible irreducible representations of $D^{\times}$ with central character $\omega$. 

Given $\sigma\in \mathcal{A}_D(\omega)$ we define the conductor exponent $a_D(\sigma)$ to be the minima of all non-negative integers $k$ such that $\sigma$ is trivial on $U_D(k)$.\footnote{In \cite{gauss_sums} the conductor is denoted by $f(\sigma)$. It is related to the conductor exponent defined here via the formula $f(\sigma)=\mathfrak{P}_D^{a_D(\sigma)}$.} If $a_D(\sigma)=0$ we call $\sigma$ unramified. By \cite[Proposition~2.1.6]{gauss_sums} all unramified representations of $D^{\times}$ are of the form $\vert \cdot \vert_D^s$ for some $s\in \C$.

\begin{rem}
Suppose that $\omega\in \mathfrak{X}$, then our assumption $\omega(\varpi)=1$ ensures that $\omega$ has finite order. As a consequence all elements in $\mathcal{A}_D(\omega)$ are finite. This is \cite[Exercise~2.1.9]{gauss_sums}.
\end{rem}

\subsection{Local factors}\label{sec:loc_fac}

The local factors associated to smooth admissible irreducible representations $\pi$ of ${\rm GL}_n(F)$ or $D^{\times}$ have been defined in \cite{godement_jacquet}. We will take this definition for granted and only record some important properties of the $\epsilon$-factors. 

We start with the case when $\pi$ is a representation of $\textrm{GL}_n(F)$. Here the $\epsilon$-factors have the following nice properties:
\begin{align}
	&\epsilon(s,\pi,\psi) = 1 \text{ if $\pi$ is unramified}, \nonumber\\
	&\epsilon(s,\pi,\psi_a) = \omega_{\pi}(a)\vert a\vert^{n(s-\frac{1}{2})}\epsilon(s,\pi,\psi) \text{ for }a\in F^{\times}, \nonumber \\
	& \epsilon(s,\abs{\cdot}^{s'}\pi,\psi) = \epsilon(s+s',\pi,\psi) \text{ for }s,s'\in \C\text{ and }\nonumber \\
	&\epsilon(s,\pi,\psi)\epsilon(1-s,\tilde{\pi},\psi) = \omega_{\pi}(-1).\nonumber
\end{align}
Furthermore, there is $f(\pi)\in \Z_{\geq 0}$ such that
\begin{equation}
	\epsilon(s,\abs{\cdot}^{s'}\pi,\psi) = q^{(n\cdot n(\psi)-f(\pi))s'}\epsilon(s,\pi,\psi). \label{eq:def_fpi}
\end{equation}
If $n=1$, then we always have $f(\chi)=a(\chi)$. This remains true for $n>1$ when restricting to generic representations. Indeed, it is shown in \cite{jacquet_cond} that for generic $\pi$ we have $a(\pi)=f(\pi)$.

\begin{rem}
Note that according to the properties postulated it suffices to define $\epsilon(\frac{1}{2},\pi,\psi)$ for $\psi$ with $n(\psi)=0$. In particular, for $n=1$ the identity \eqref{eq:gl_2_rootnb} for $\chi\in \mathfrak{X}$ is sufficient to completely determine all $\textrm{GL}_1$ $\epsilon$-factors.
\end{rem}

We turn towards the local factors associated to representations $\sigma\in \mathcal{A}_D(\omega)$. These were also defined in \cite{godement_jacquet} and satisfy the same properties. In particular, we can use (the analogous version of) \eqref{eq:def_fpi} to define $f(\sigma)$. According to \cite[Lemma~3.2]{corbett_twist} we have the relation
\begin{equation}
	f(\sigma) = a_D(\sigma)+n-1. \label{eq:com_conducs}
\end{equation}

For representations of $D^{\times}$ we can now generalize the root number construction from \eqref{eq:gl_2_rootnb}. This is done as follows. Given $\sigma\in \mathcal{A}_D(\omega)$ we first associate the non-abelian congruence Gau\ss\  sum $\tau(\sigma)$ as in \cite[Proposition~2.2.4]{gauss_sums}. Note that one can show
\begin{equation}
	\vert\tau(\sigma)\vert^2=q^{n\cdot a_D(\sigma)}. \nonumber
\end{equation}
This naturally leads to the definition of the root number
\begin{equation*}
	W(\sigma) = (-1)^{n+1}\omega(-1)\overline{\tau(\sigma)}q^{-\frac{n}{2}a_D(\sigma)} \nonumber
\end{equation*}
as in \cite[(2.4.8)]{gauss_sums}.

\begin{rem}\label{rem:justified}
If $\sigma=\chi\circ\textrm{Nrd}$ for a quasi-character $\chi$ of $F^{\times}$, then  \cite[Theorem~4.1.5]{gauss_sums}  implies
\begin{equation}
	W(\sigma) = W(\chi)^n. \nonumber
\end{equation}
This justifies the definition for characters made in \eqref{eq:imp_rel_chi_norm}.
\end{rem}

Before we turn towards the relation between root numbers and $\epsilon$-factors we record the following important result, which can be thought of as a vast generalization of Lemma~\ref{lm_gl1_stab}:
\begin{lemmy}[Corollary~2.5.11, \cite{gauss_sums}]\label{lm_D_stab}
Let $\sigma\in \mathcal{A}_D(\omega)$ and let $\xi$ be a character of $D^{\times}$. Suppose that $0<2a_D(\sigma)\leq a_D(\xi)$ and that there is $c\in F^{\times}$ with
\begin{equation}
	\xi(1+x) = \psi_D(c\cdot y) \text{ for }x\in \mathfrak{P}_D^{a(\xi)-a(\sigma)}.\label{eq:condition_char}
\end{equation}
Then we have
\begin{equation*}
	W(\xi\otimes \sigma) = \omega(c_{\xi})\cdot W(\xi).
\end{equation*}
\end{lemmy}

Finally, we can express the $\epsilon$-factors for $\sigma\in \mathcal{A}_D(\omega)$ in terms of root numbers. According to \cite[Theorem~3.2.11]{gauss_sums} we have
\begin{equation}
	\epsilon(s,\sigma,\psi)= q^{(s-\frac{1}{2})[n\cdot n(\psi)-n+1-a_D(\sigma)]}	W(\sigma), \text{ for }\sigma\in \mathcal{A}_D(\omega). \label{imp_equalit_eps}
\end{equation}

We are now ready to state the following key lemma:

\begin{lemmy}\label{lm:finally}
Let $\omega\in \mathfrak{X}$ and $\sigma\in \mathcal{A}_{D}(\omega)$ ramified. If $\chi$ be a quasi-character of $F^{\times}$ satisfying 
\begin{equation*}
	na(\chi)-n+1\geq 2a_D(\sigma),
\end{equation*}
then we have
\begin{equation*}
	\epsilon(\frac{1}{2},(\chi\circ\textrm{Nrd})\otimes \sigma,\psi) = \omega(v_{\chi})\cdot \epsilon(\frac{1}{2},\chi,\psi)^n 
\end{equation*}
where $v_{\chi}$ depends only on $\chi$ (and $\psi$).
\end{lemmy}
\begin{proof}
	Our assumption ensures that we can apply Lemma~\ref{lm_D_stab} with $c=v_{\chi}\varpi^{n(\psi)-\lfloor \frac{a(\chi)}{2}\rfloor}$ for $v_{\chi}\in \mathcal{O}^{\times}$ as in \eqref{eq:additiviation}. We obtain
	\begin{equation*}
		\epsilon(\frac{1}{2},(\chi\circ\textrm{Nrd})\otimes \sigma,\psi) = W((\chi\circ\textrm{Nrd})\otimes \sigma) = \omega(v_{\chi})\cdot W(\chi\circ\textrm{Nrd}).\nonumber
	\end{equation*}
	We conclude by applying \eqref{eq:imp_rel_chi_norm} and recalling \eqref{eq:gl_2_rootnb}.
\end{proof}

\section{The main proof}\label{sec:proof}

We are now ready to prove  Theorem~\ref{main_th}. We will do so in three steps. First, we consider supercuspidal representations. Here we will employ the Jacquet-Langlands correspondence to reduce to the division algebra case. Second we upgrade this to (essentially) square integrable representations using their classification. We conclude by treating the general case of generic representation using parabolic induction.

\subsection{The supercuspidal case}

Recall that supercuspidal representations are always generic. Thus, in this section our goal is to prove the following:

\begin{prop}\label{prop:sc}
Suppose $\pi$ is a supercuspidal representation of $\textrm{GL}_n(F)$ with $n\geq 2$. Then we have
\begin{equation}
	\epsilon(s,\chi\cdot\pi,\psi) = \omega_{\pi}(v_{\chi}\varpi^{a(\chi)-n(\psi)})\cdot \epsilon(s,\chi,\psi)^n, \label{eq:prop_sc}
\end{equation}
for every quasi character $\chi\colon F^{\times}\to \C^{\times}$ satisfying
\begin{equation}
	a(\chi) \geq \frac{2a(\pi)+1}{n}-1. \nonumber
\end{equation}
Here $v_{\chi}$ is as in \eqref{eq:def_v_chi}. 
\end{prop}
\begin{proof}
We start by making some reductions. Recall from \cite[Proposition~2.2]{corbett_twist} that, if $a(\chi)>\frac{a(\pi)}{n}$, then we have $a(\chi\cdot \pi) = n\cdot a(\chi)$. Furthermore, we write 
\begin{equation*}
	\pi = \vert \cdot \vert^{s_{\pi}}\cdot \pi'
\end{equation*}
such that $\pi'$ has central character $\omega_{\pi'}\in \mathfrak{X}$. Similarly we write $\chi = \vert \cdot \vert^{s_{\chi}}\cdot \chi'$ with $\chi'\in \mathfrak{X}$. As soon as $a(\chi)>\frac{a(\pi)}{n}$ we have
\begin{equation*}
	\epsilon(s,\chi\cdot\pi,\psi) = q^{n((\psi)-a(\chi))(s+s_{\pi}+s_{\chi}-\frac{1}{2})}\epsilon(\frac{1}{2},\chi'\cdot\pi',\psi). \nonumber
\end{equation*}
On the other hand we have
\begin{equation*}
	\omega_{\pi}(v_{\chi}\varpi^{a(\chi)-n(\psi)})\epsilon(s,\chi,\psi)^n = \omega_{\pi'}(v_{\chi})\epsilon(\frac{1}{2},\chi',\psi)^n\cdot q^{n((\psi)-a(\chi))(s+s_{\pi}+s_{\chi}-\frac{1}{2})}
\end{equation*}
We conclude that it is sufficient to prove \eqref{eq:prop_sc} for $s=\frac{1}{2}$, $\chi\in \mathfrak{X}$ and $\pi$ with central character in $\mathfrak{X}$.

We will now use the Jacquet-Langlands correspondence between supercuspidal representations of $\textrm{GL}_n(F)$ and representations of $D^{\times}$. In our case this is essentially due to Rogawski in \cite{rog}. However, all the properties we need are very conveniently stated in \cite[Theorem~2.2]{abps}. Recall that the Jacquet-Langlands correspondence allows us to take our supercuspidal representation $\pi$ and associate a ramified representation
\begin{equation*}
	\textrm{JL}(\pi) \in \mathcal{A}_D(\omega_{\pi}).
\end{equation*}
This assignment has many nice properties. Most importantly
\begin{equation*}
	\textrm{JL}(\chi\cdot\pi) = (\chi\circ\textrm{Nrd})\otimes \textrm{JL}(\pi) 
\end{equation*}
and
\begin{equation*}
	\epsilon(s,\pi,\psi) = \epsilon(s,\textrm{JL}(\pi),\psi). \nonumber
\end{equation*}
In particular $f(\pi) = f(\textrm{JL}(\pi))$, so that by \eqref{eq:com_conducs} we have
\begin{equation*}
	a_D(\textrm{JL}(\pi)) = a(\pi)-n+1.
\end{equation*}
The assumption of our theorem is set-up such that $na(\chi)-n+1\geq 2a_D(\textrm{JL}(\pi))$. An application of Lemma~\ref{lm:finally} with $\sigma=\textrm{JL}(\pi)$ concludes the proof. 
\end{proof}

\subsection{Square integrable representations} 

Essentially square integrable representations have been classified in \cite[Theorem~9.3]{zelevisnky_2} for example. They are all generic. To state this classification we introduce the special representation ${\rm St}(\tau,d)$ as the unique irreducible submodule of the induced representation
\begin{equation}
	{\rm Ind}_{P(F)}^{{\rm GL}_n(F)}(\abs{\cdot}^{\frac{d-1}{2}}\cdot\tau \otimes \ldots \otimes \abs{\cdot}^{-\frac{d-1}{2}}\cdot \tau), \nonumber
\end{equation}
where $d\mid n$, $\tau$ is an essentially supercuspidal representation of ${\rm GL}_{\frac{n}{d}}(F)$ and $P$ is the parabolic subgroup associated to the partition $(\frac{n}{d},\ldots,\frac{n}{d})$ of $n$. We observe that $\tau = {\rm St}(\tau,1)$ is essentially supercuspidal. Furthermore, if $n=1$, then $\tau=\chi$ is a quasi-character and ${\rm St}(\tau,n) = \chi\cdot {\rm St}_n$ is (a twist of) the usual Steinberg representation, explaining the notation. The classification now states that any essentially square integrable representation $\pi$ is isomorphic to ${\rm St}(\tau,d)$ for some $d\mid n$ and some essentially supercuspidal representation $\tau$ of ${\rm GL}_{\frac{n}{d}}(F)$. Thus the following result covers all essentially square integrable representations:

\begin{prop}\label{pr:sq_i}
Let $d\mid n$ and let $\tau$ be an essentially supercuspidal representation of $\textrm{GL}_{\frac{n}{d}}(F)$. If $\chi$ is a ramified quasi-character of $F^{\times}$ with 
\begin{equation*}
	a(\chi) \geq \frac{2da(\tau)+d}{n}-1,
\end{equation*}
then we have
\begin{equation*}
	\epsilon(s,\chi\cdot\textrm{St}(\tau,d),\psi) = \omega_{\textrm{St}(\tau,d)}(v_{\chi}\varpi^{a(\chi)-n(\psi)})\cdot \epsilon(s,\chi,\psi)^n.
\end{equation*}
\end{prop}
\begin{proof}
If $d=1$, then $\textrm{St}(\tau,d)=\tau$ is supercuspidal and the result follows from Proposition~\ref{prop:sc}.

Next we assume that $1<d<n$. Then we have $a(\textrm{St}(\tau,d)) = d\cdot a(\tau)$, $\omega_{\textrm{St}(\tau,d)} = \omega_{\tau}^d$ and
\begin{equation*}
	\epsilon(s,\chi\cdot\textrm{St}(\tau,d),\psi) = \epsilon(s,\textrm{St}(\chi\cdot \tau,d),\psi) = \epsilon(s,\chi\cdot \tau)^d.
\end{equation*}
In particular, if $a(\chi)\geq \frac{a(\tau)+1}{n/d}-1$, then Proposition~\ref{prop:sc} implies
\begin{equation*}
	\epsilon(s,\chi\cdot\textrm{St}(\tau,d),\psi) = [\omega_{\tau}(v_{\chi}\varpi^{a(\chi)-n(\psi)})\epsilon(s,\chi,\psi)^{\frac{n}{d}}]^d.
\end{equation*}
This is precisely the statement we are after.

Finally, we treat the case $d=n$, so that $\tau$ is a quasi-character of $F$. The usual reduction shows that it is sufficient to consider $\tau,\chi\in \mathfrak{X}$. Since we are assuming $\chi$ to be ramified we have 
\begin{equation*}
		\epsilon(s,\chi\cdot\textrm{St}(\tau,d),\psi) = \epsilon(s,\chi\tau,\psi)^n \nonumber
\end{equation*}
and our condition on the conductor reduces to $2a(\tau)\leq a(\chi)$. We conclude by an application of Lemma~\ref{lm_gl1_stab}.
\end{proof}

\begin{rem}
Note that Proposition~\ref{pr:sq_i} also applies to the case $n=1$, where it reduces essentially to Lemma~\ref{lm_gl1_stab}.	
\end{rem}

\subsection{The general case}

By the Langlands classification, all smooth admissible irreducible representations can be described through induced representations. Let $P(F)$ be a standard parabolic subgroup associated to the partition $(n_1,\ldots, n_r)$ of $n$ and let $\pi_1,\ldots, \pi_r$ be essentially supercuspidal representations of ${\rm GL}_{n_i}(F)$, then the representation $${\rm Ind}_{P(F)}^{{\rm GL}_n(F)}(\abs{\cdot }^{a_1}\pi_1\otimes \ldots \otimes  \abs{\cdot}^{a_r}\pi_r)$$ obtained by normalized parabolic induction has a unique irreducible quotient denoted by $J(\abs{\cdot }^{a_1}\pi_1,\ldots,\abs{\cdot }^{a_r}\pi_r)$ and called the Langlands quotient. It can be shown, see \cite[Theorem~6.1]{zelevisnky_2}, that every smooth irreducible admissible representation of ${\rm GL}_n(F)$ is isomorphic to $J(\abs{\cdot }^{a_1}\pi_1,\ldots,\abs{\cdot }^{a_r}\pi_r)$ for a uniquely determined multiset $\{\abs{\cdot }^{a_1}\pi_1,\ldots,\abs{\cdot }^{a_r}\pi_r\}$. Note that in this notation
\begin{equation}
	{\rm St}(\tau,d) = J(\abs{\cdot}^{-\frac{d-1}{2}}\cdot\tau \otimes \ldots \otimes \abs{\cdot}^{\frac{d-1}{2}}\cdot \tau). \nonumber
\end{equation}
About generic representations more can be said. First we introduce so called representations of Langlands type. These are induced representations $${\rm Ind}_{P(F)}^{{\rm GL}_n(F)}(\abs{\cdot }^{a_1}\pi_1\otimes \ldots \otimes \abs{\cdot}^{a_r}\pi_r)$$ as above with the condition that the representations $\pi_1,\ldots,\pi_r$ are irreducible unitary and tempered and the exponents $a_1,\ldots a_r$ are real and satisfy $a_1\geq a_2\geq \ldots \geq a_r$. We can specialize these induced further, to so called representations of Whittaker type, by requiring that the representations $\pi_1,\ldots,\pi_r$ are quasi square integrable. Note that representations of Whittaker type are not necessarily irreducible, but it is well understood by \cite[Theorem~9.7]{zelevisnky_2} when they are. The same theorem, establishes that every generic irreducible admissible smooth representation $\pi$ of ${\rm GL}_n(F)$ is isomorphic to an irreducible representation of Whittaker type:
\begin{equation}
	\pi \cong {\rm Ind}_{P(F)}^{{\rm GL}_n(F)}(\abs{\cdot }^{a_1}\pi_1\otimes \ldots \otimes \abs{\cdot}^{a_r}\pi_r), \nonumber
\end{equation}
for square integrable representations $\pi_1,\ldots,\pi_r$ and real numbers $a_1\geq \ldots\geq a_r$. The multiset $\{\abs{\cdot}^{a_1}\pi_1,\ldots,\abs{\cdot}^{a_r}\pi_r\}$ is uniquely determined by $\pi$. 

We are ready to prove the following central result:
\begin{prop}\label{pr:gen}
Let $\pi = {\rm Ind}_{P(F)}^{{\rm GL}_n(F)}(\abs{\cdot }^{a_1}\pi_1\otimes\ldots \otimes \abs{\cdot}^{a_r}\pi_r)$ be a representation of Whittaker type and suppose that $\pi_i=\textrm{Sp}(\tau_i,d_i)$. If $\chi$ is a quasi-character of $F^{\times}$ with
\begin{equation*}
	a(\chi) \geq \max_i \frac{2d_1a(\tau_i)+d_i}{n_i}-1,
\end{equation*}
then we have
\begin{equation*}
	\epsilon(s,\chi\cdot \pi,\psi) = \omega_{\pi}(v_{\chi}\varpi^{a(\chi)-n(\psi)}) \cdot \epsilon(s,\chi,\psi)^n.
\end{equation*}
\end{prop}
\begin{proof}
We observe that 
\begin{equation*}
	\epsilon(s,\chi\cdot \pi,\psi) = \prod_i \epsilon(s,\vert\cdot\vert^a_i\chi\cdot \pi_i,\psi).\nonumber
\end{equation*}
Our assumption on the conductor exponent of $\chi$ ensures that we can apply Proposition~\ref{pr:sq_i} for each $i$. Since $n_1+\ldots+n_r=n$ we have
\begin{equation*}
	\epsilon(s,\chi\cdot \pi,\psi) = \left(\prod_i \omega_{\pi_i}(v_{\chi}\varpi^{a(\chi)-n(\psi)})\vert \varpi^{a(\chi)-n(\psi)}\vert^{a_i} \right)\epsilon(s,\chi,\psi)^n.
\end{equation*}
Recognizing the remaining product as the central character $\omega_{\pi}$ of $\pi$ evaluated at $v_{\chi}\varpi^{a(\chi)-n(\psi)}$ completes the proof.
\end{proof}

\begin{proof}[Proof of Theorem~\ref{main_th}]
Suppose $\pi$ is generic and let $\chi$ be a quasi-character of $F^{\times}$ with $a(\chi)\geq a(\pi)$. According to our classification above we can write
\begin{equation*}
	\pi = {\rm Ind}_{P(F)}^{{\rm GL}_n(F)}(\pi_1\otimes \ldots\otimes \pi_r),
\end{equation*}
for essentially square integrable representations $\pi_i$. 

First we suppose that
\begin{equation}
	a(\chi)\geq \frac{2}{n_i}a(\pi_i) \label{first_ass}
\end{equation}
for all $i$. Then we can apply Proposition~\ref{pr:sq_i} to find
\begin{equation}
	\epsilon(s,\chi\cdot\pi_i,\psi) = \omega_{\pi_i}(v_{\chi}\varpi^{a(\chi)-n(\psi)})\epsilon(s,\chi,\psi)^{n_i}. \nonumber
\end{equation}
In particular, taking the product over all $i$ yields
\begin{equation*}
	\epsilon(s,\pi,\psi) = \omega_{\pi}(v_{\chi}\varpi^{a(\chi)-n(\psi)})\epsilon(s,\chi,\psi)^n = \epsilon(s,\omega_{\pi}\chi,\psi)\cdot \epsilon(s,\chi,\psi)^{n-1}.
\end{equation*}
In the last step we have applied Lemma~\ref{lm_gl1_stab} and \eqref{eq:def_fpi}.

It remains to consider the possibility that \eqref{first_ass} fails for some $i\in \{1,\ldots,r\}$. Note that, since $a(\chi)\geq a(\pi)\geq a(\pi_i)$ this can only happen if $n_i=1$. Next we claim that this can only happen for at most one $1\leq i\leq r$. Indeed, suppose that $a(\chi) \leq 2a(\pi_i)$ and $a(\chi)\leq 2a(\pi_j)$ with $i\neq j$. Then 
\begin{equation*}
	a(\chi) \leq a(\pi_i)+ a(\pi_j) \leq a(\pi).
\end{equation*}
This is a contradiction. Without loss of generality we assume that $n_1=1$, $a(\chi)< 2a(\pi_1)$ and $a(\chi)\geq \frac{2}{n_i}a(\pi_i)$ for $i=2,\ldots,r$. Put $\omega'=\prod_{i=2}^r\omega_{\pi_i}$. Using Proposition~\ref{pr:sq_i} once again allows us to write
\begin{equation*}
	\epsilon(s,\chi\cdot\pi,\psi) = \omega'(v_{\chi}\varpi^{a(\chi)-n(\psi)})\epsilon(s,\chi\cdot\pi_1,\psi)\cdot \epsilon(s,\chi,\psi)^{n-1}.
\end{equation*}
If $\pi_i$ is ramified for at least one $i\in \{2,\ldots, r\}$, then we claim that 
\begin{equation}
	2a(\omega')\leq a(\chi\cdot \pi_1). \label{last_claim}	
\end{equation}
Indeed, if this is the case, then an application of Lemma~\ref{lm_gl1_stab} yields
\begin{equation*}
	\omega'(v_{\chi}\varpi^{a(\chi)-n(\psi)})\epsilon(s,\chi\cdot \pi_1,\psi) = \epsilon(s,\omega_{\pi}\chi,\psi)
\end{equation*}
and we are done.

To see \eqref{last_claim} we argue as follows. We first observe that the assumption ensures that $a(\chi)\geq a(\pi)>a(\pi_1)$ so that $a(\chi\cdot\pi_1)=a(\chi)$. On the other hand, we can use \eqref{eq:useful_cond} to find
\begin{equation*}
	a(\omega') \leq \max_{j=2,\ldots,r}a(\omega_{\pi_j})\leq \max_{j=2,\ldots,r}\frac{a(\pi_j)}{n_j} \leq \frac{a(\chi)}{2}.
\end{equation*}
This establishes the claim.

The slightly exceptional case when $\pi_2,\ldots,\pi_r$ are all unramified is easily treated by hand. This completes the proof. 
\end{proof}

\section{An explicit Bessel transform} \label{sec:bessel}

In this section we indicate how the stability result can be used to compute certain Bessel transforms. We will start by summarizing the assumptions we will make. Throughout we will let $\pi$ be a complex generic irreducible admissible smooth representation of $\textrm{GL}_n(F)$ such that
\begin{equation*}
	\max(1,a(\omega_{\pi})) <a(\pi).
\end{equation*}
We further assume that $\omega_{\pi}\in \mathfrak{X}$ and $n(\psi)=0$.

Recall from \cite[Lemma~5.2]{ichino_templier} or \cite[Propositon~3.1]{corbett_gln} that the Bessel transform $\mathcal{B}_{\pi}\Phi$ of $\Phi\in \mathcal{C}_c^{\infty}(F^{\times})$ is defined by the duality relations
\begin{multline}
	\int_{F^{\times}}[\mathcal{B}_{\pi}\Phi](y)\chi(y)^{-1}\vert y\vert^{s-\frac{n-1}{2}}d^{\times}y \\ = \chi(-1)^{n-1}\epsilon(1-s,\chi\pi,\psi)\frac{L(s,\chi^{-1}\widetilde{\pi})}{L(1-s,\chi\pi)} \int_{F^{\times}}\chi(x)\Phi(x)\vert x\vert^{1-s-\frac{n-1}{2}}d^{\times}x. \label{eq:bessel_trafro}
\end{multline}
This can be Mellin inverted and we obtain the exact formula
\begin{multline*}
	[\mathcal{B}_{\pi}\Phi](y) = \frac{\log(q)}{2\pi}\sum_{\chi\in \mathfrak{X}}\chi(-1)^{n-1}\chi(y)\int_{\sigma-i\frac{\pi}{\log(q)}}^{\sigma+i\frac{\pi}{\log(q)}} \epsilon(1-s,\chi\pi,\psi)\frac{L(s,\chi^{-1}\widetilde{\pi})}{L(1-s,\chi\pi)}\vert y\vert^{\frac{n-1}{2}-s} \\ \cdot \int_{F^{\times}}\Phi(x)\chi(x)\vert x\vert^{1-s-\frac{n-1}{2}}d^{\times}xds
\end{multline*}
as in \cite[(34)]{corbett_gln}. Note that here we normalize $d^{\times}x$ such that $\mathcal{O}^{\times}$ has volume $1$.

We define the function
\begin{equation}
	\Phi_z(x) = \psi(zx)\cdot\mathbbm{1}_{\mathcal{O}^{\times}}(x) \text{ for }z\in F^{\times} \text{ with }v(z) \leq -a(\pi). \label{Phi}
\end{equation}
In this case the $x$-integral above is a simple Gau\ss\  sum which we can evaluate using \cite[Lemma~4.5]{corbett_gln} for example. We record the following result in spirit of \cite[Proposition~4.7]{corbett_gln}:

\begin{lemmy}\label{prelim_eval}
Let $\pi$ be as above and let $\Phi_z$ be as in \eqref{Phi}. Then $[\mathcal{B}_{\pi}\Phi_z](y)=0$ unless $v(y)=nv(z)$ in which  case
\begin{equation*}
	[\mathcal{B}_{\pi}\Phi_z](y) = \vert y\vert^{\frac{(n-1)^2}{2n}}\cdot \frac{\vert z\vert^{-1}}{1-q^{-1}}
	 \sum_{\substack{\chi\in \mathfrak{X}, \\a(\chi)=-v(z)}} \chi((-1)^{n-1}yz^{-1})\epsilon(\frac{1}{2},\chi^{-1},\psi)\epsilon(\frac{1}{2},\chi\cdot\pi,\psi).
\end{equation*}
\end{lemmy}
\begin{proof}
We first note that
\begin{align*}
	\int_{F^{\times}}\Phi_z(x)\chi(x)\vert x\vert^{1-s-\frac{n-1}{2}}d^{\times}x &=  \int_{\mathcal{O}^{\times}}\psi(zx)\chi(x)d^{\times}x \\
	&= \delta_{a(\chi)=-v(z)}\frac{1}{1-q^{-1}} \vert z\vert^{-\frac{1}{2}}\chi(z)^{-1}\epsilon(\frac{1}{2},\chi^{-1},\psi).
\end{align*}	
Here we note that our assumptions imply $-v(z)\geq a(\pi)\geq 2$, so that there is no contribution from the trivial character.

Next notice that, since $a(\omega_{\pi})<a(\pi)\leq -v(z)=a(\chi)$, we have $a(\chi\pi)=na(\chi)$. We can thus rewrite the Bessel transform of $\Phi_z$ as
\begin{multline*}
	[\mathcal{B}_{\pi}\Phi_z](y) = \frac{1}{1-q^{-1}}\vert z\vert^{-\frac{n+1}{2}}\vert y\vert^{\frac{n-1}{2}} \sum_{\substack{\chi\in \mathfrak{X}, \\a(\chi)=-v(z)}} \chi((-1)^{n-1}yz^{-1})\epsilon(\frac{1}{2},\chi^{-1},\psi)\epsilon(\frac{1}{2},\chi\cdot\pi,\psi)  \\ \cdot \frac{\log(q)}{2\pi}\int_{\sigma-i\frac{\pi}{\log(q)}}^{\sigma+i\frac{\pi}{\log(q)}} \frac{L(s,\chi^{-1}\widetilde{\pi})}{L(1-s,\chi\pi)}q^{na(\chi)s}\vert y\vert^{-s}ds.
\end{multline*}
Once again $a(\omega_{\pi})<a(\pi)$ ensures that $$L(1-s,\chi\pi) = L(s,\chi^{-1}\widetilde{\pi})=1$$ for all $\chi$ with $a(\chi) \geq a(\pi)$. Thus we have obtained
\begin{equation*}
		[\mathcal{B}_{\pi}\Phi_z](y) = \frac{\delta_{nv(z)=v(y)}}{1-q^{-1}}\vert z\vert^{-\frac{n+1}{2}} \vert y\vert^{\frac{n-1}{2}}\sum_{\substack{\chi\in \mathfrak{X}, \\a(\chi)=-v(z)}} \chi((-1)^{n-1}yz^{-1})\epsilon(\frac{1}{2},\chi^{-1},\psi)\epsilon(\frac{1}{2},\chi\cdot\pi,\psi).
\end{equation*}
Reformulating this slightly gives the desired result.
\end{proof}

Finally, define the (twisted) Hyperkloosterman sum as
\begin{equation*}
	\textrm{KL}_{\omega_{\pi},n}(y;t) = \sum_{x_1\in (\mathcal{O}/\mathfrak{p}^t)^{\times} }\cdots \sum_{x_{n-1}\in (\mathcal{O}\mathfrak{p}^{t})^{\times}}\omega_{\pi}(x_1)\psi\left(\varpi^{-t}\left[x_1+\ldots+x_{n-1}+\frac{y}{x_1\cdots x_{n-1}}\right]\right).
\end{equation*}

\begin{prop}\label{prop:bessel}
Let $\pi$ be a complex generic irreducible smooth admissible representation of $\textrm{GL}_n(F)$ with $\max(1,a(\omega_{\pi}))<a(\pi)$. For $z\in F^{\times}$ with $-v(z)\geq a(\pi)$ and $\Phi_z$ as in \eqref{Phi} we have
\begin{equation*}
	[\mathcal{B}_{\pi}\Phi_z](y) = \frac{1}{(1-q^{-1})^{n-1}}\vert y\vert^{\frac{(n-4)(n-1)}{2n}} \begin{cases}
		\textrm{KL}_{\omega_{\pi}^{-1},n-1}(a(y,z),-v(z)) &\text{ if }v(y)=nv(z),\\
		0&\text{ else}
	\end{cases}
\end{equation*}
for $a(y,z)=(-1)^nyz^{-1}\varpi^{-(n-1)v(z)}$.
\end{prop}
\begin{proof}
According to Lemma~\ref{prelim_eval} we can assume that $v(y)=nv(z)$. Due to the assumption $a(\pi)\leq -v(z)$ we can apply Theorem~\ref{main_th} and obtain
\begin{equation*}
	[\mathcal{B}_{\pi}\Phi_z](y) = \vert y\vert^{\frac{(n-1)^2}{2n}}\cdot \frac{\vert z\vert^{-1}}{1-q^{-1}}\sum_{\substack{\chi\in \mathfrak{X}, \\a(\chi)=-v(z)}} \chi((-1)^{n}yz^{-1})\epsilon(\frac{1}{2},\chi,\psi)^{n-2}\epsilon(\frac{1}{2},\chi\omega_{\pi},\psi).
\end{equation*}
At this point we can express the $\epsilon$-factors as Gau\ss\  sums as in \eqref{eq:gl_2_rootnb}. This gives
\begin{multline*}
	[\mathcal{B}_{\pi}\Phi_z](y) = \vert y\vert^{\frac{(n-2)(n-1)}{2n}}\cdot \frac{\vert z\vert^{-1}}{1-q^{-1}}\sum_{\substack{\chi\in \mathfrak{X}, \\a(\chi)=-v(z)}}\chi((-z)^{-n}y)\omega_{\pi}(z)^{-1}\\ \int_{\mathcal{O}^{\times}}\cdots \int_{\mathcal{O}^{\times}} \omega_{\pi}(x_1)^{-1}\chi(x_1\cdots x_{n-1})^{-1}\psi(z[x_1+\ldots+x_{n-1}])dx_1\cdots dx_{n-1}.
\end{multline*}
Using $-v(z)\geq 2$ we can now relax the condition $a(\chi)=-v(z)$ to $a(\chi)\leq -v(z)$ without picking up any extra contribution from the trivial character. Executing the $\chi$-sum then yields
\begin{multline*}
	[\mathcal{B}_{\pi}\Phi_z](y) = \vert y\vert^{\frac{(n-2)(n-1)}{2n}} \underset{x_1\cdots x_{n-1}(-z)^{n}y^{-1}\in 1+\varpi^{-v(z)}\mathcal{O}}{\int_{\mathcal{O}^{\times}}\cdots \int_{\mathcal{O}^{\times}}} \omega_{\pi}(zx_1)^{-1} \\ \cdot \psi(z[x_1+\ldots+x_{n-1}])dx_1\cdots dx_{n-1}.
\end{multline*}
Here we have used that $\sharp \{ \chi\in \mathfrak{X}\colon a(\chi)\leq -v(z)\} = \vert z\vert(1-q^{-1})$. Writing for the moment $z=z_0\cdot\varpi^{v(z)}$ and $y=y_0\varpi^{nv(z)}$ allows us to rewrite this as
\begin{multline*}
	[\mathcal{B}_{\pi}\Phi_z](y) = \vert y\vert^{\frac{(n-2)(n-1)}{2n}}\frac{\vert z\vert^{-1}}{1-q^{-1}} \\ \int_{\mathcal{O}^{\times}}\cdots \int_{\mathcal{O}^{\times}} \omega_{\pi}(x_1)^{-1}\psi\left(\varpi^{v(z)}\left[x_1+\ldots+x_{n-2}+\frac{(-1)^ny_0z_0^{-1}}{x_1\cdots x_{n-2}}\right]\right)dx_1\cdots dx_{n-2}.
\end{multline*}
Replacing the $\mathcal{O}^{\times}$-integrals by the appropriate sums gives the desired result.
\end{proof}

\begin{rem}
The assumptions such as $a(\pi)>\max(1,a(\omega_{\pi}))$ is purely of cosmetic nature. Indeed, removing it requires some more bookkeeping. Since the general result does not look as clean we have decided to omit the details.
\end{rem}

\end{document}